\documentclass[oneside,12pt]{amsart}
\usepackage{eurosym}
\usepackage{amsfonts}
\usepackage{amsmath,amssymb,amsthm,color}
\usepackage{amsopn}
\usepackage{latexsym}
\usepackage{float}
\usepackage{bm}
\usepackage[utf8]{inputenc}

\newcommand{\mg}{\mathfrak g }

\newcommand{\ms}{\mathfrak s }
\newcommand{\mn}{\mathfrak n }

\newcommand{\mz}{\mathfrak z }

\newcommand{\mv}{\mathfrak v }
\newcommand{\mh}{\mathfrak h }

\newcommand{\bil}{g}
\newcommand{\lela}{ \left\langle}
\newcommand{\rira}{\right\rangle}
\newcommand{\lra}{\longrightarrow}

\newcommand{\mcQ}{\mathcal Q}

\renewcommand{\Im}{{\rm Im}}
\newcommand{\R}{\mathbb R}

\newcommand{\Z}{\mathbb Z}

\DeclareMathOperator{\End}{End}

\DeclareMathOperator{\ad}{ad}

\DeclareMathOperator{\tr}{tr}

\numberwithin{equation}{section}

 \newtheorem{teo}{Theorem}[section]
 \newtheorem{pro}[teo]{Proposition}
 \newtheorem{cor}[teo]{Corollary}
 \newtheorem{lm}[teo]{Lemma}
 \newtheorem{defi}[teo]{Definition}

 \theoremstyle{definition}
 \newtheorem{ex}[teo]{Example}
 \newtheorem{remark}[teo]{Remark}

\newcommand{\nc}{\newcommand}
\nc{\Iso}{\operatorname{Iso}}
 \nc{\iso}{\mathfrak{iso}}
 \nc{\sso}{\mathfrak{so}}
\nc{\Ad}{\operatorname{Ad}} 
\nc{\Sym}{\mathrm{Sym}}
  \nc{\pr}{\operatorname{pr}} 
 \nc{\Dera}{\operatorname{Dera}} \nc{\Auto}{\operatorname{Auto}}
 \nc{\LL}{{\rm L}}
\nc{\dd}{{\rm d}}
\nc{\Id}{{\rm Id}}

\begin{document}

\title{Conformal Killing symmetric tensors on Lie groups}
\author{Viviana del Barco}
\address{Universidad Nacional de Rosario, CONICET, 2000, Rosario, Argentina}
\email{delbarc@fceia.unr.edu.ar}

\author{Andrei Moroianu}
\address{Université Paris-Saclay, CNRS,  Laboratoire de mathématiques d'Orsay, 91405, Orsay, France}
\email{andrei.moroianu@math.cnrs.fr}

\begin{abstract} We introduce the notion of metric Lie algebras of Killing type, which are characterized by the fact that all conformal Killing symmetric tensors are sums of Killing tensors and multiples of the metric tensor. We show that if a Lie algebra is either 2-step nilpotent, or 2- or 3-dimensional, or 4-dimensional non-solvable, or 4-dimensional solvable with 1-dimensional derived ideal, or has an abelian factor, then it is of Killing type with respect to any positive definite metric. 
\end{abstract}

\maketitle


\section{Introduction}

Symmetric Killing tensors on (pseudo-)Riemannian manifolds are symmetric tensors whose symmetrized covariant derivative vanishes. They define first integrals of the geodesic flow on the tangent bundle, they are polynomial in the momenta, and were first considered in the physics literature, see e.g. \cite{penrose}, \cite{wood}. 

More generally, one can define conformal Killing tensors as symmetric tensors whose symmetrized covariant derivative is the symmetric product of the metric with some other symmetric tensor. They have the remarkable property of being conformally invariant \cite[Lemma 3.2]{HMS16} and still define polynomial first integrals for null geodesics.

Recently, Killing and conformal Killing symmetric tensors also appeared in the framework of geometric inverse problems \cite{salo2}, \cite{salo1}, integrable systems \cite{mat11} and Riemannian geometry \cite{coll}, \cite{dairbekov}, \cite{HMS16}, \cite{HMS17}. 

In order to explain in more detail the relationship between Killing and conformal Killing tensors, let us denote by $(M,g)$ a Riemannian manifold, by $\dd:\Sym^p \mathrm{T}M\lra \Sym^{p+1} \mathrm{T}M$ the symmetrized covariant derivative and by $\LL : \Sym^p \mathrm{T}M\lra \Sym^{p+2} \mathrm{T}M$ the product with the metric tensor. Then $K\in \Gamma(\Sym^p \mathrm{T}M)$ is Killing if and only if $\dd K=0$, and conformal Killing if and only if $\dd K\in \Im(\LL)$. 

Since the operators $\dd$ and $\LL$ commute, it turns out that if $K\in \Gamma(\Sym^p \mathrm{T}M)$ is Killing, then $K+\LL R$ is conformal Killing for every $R\in \Gamma(\Sym^{p-2} \mathrm{T}M)$. The conformal Killing tensors obtained in this way are called of Killing type. We are interested in the existence of genuine conformal Killing tensors, i.e. which are not of Killing type. For instance, any conformal vector field on a Riemannian manifold which is not a Killing vector field provides a genuine symmetric tensor in this sense.

In the present paper we will study this problem for left-invariant symmetric tensors on Riemannian Lie groups, where it can be interpreted in terms of an algebraic problem on the corresponding metric Lie algebras. The paper is organized as follows.

In Section 2 we review Killing and conformal Killing symmetric tensors on Riemannian manifolds, with a special focus on Riemannian Lie groups. As a preliminary result, we show that every left-invariant conformal vector field is Killing, so the genuine example provided above does not apply in our case. We introduce the notion of metric Lie algebra of Killing type (defined by the fact that all the left-invariant conformal Killing tensors on the corresponding Lie group are of Killing type). In Section 3 we use a particular decomposition associated to 2-step nilpotent Lie algebras to show, in Theorem \ref{teo:2nilp}, that they are always of Killing type, for any possible metric. The proof is based on an inductive argument explained in Proposition \ref{pro:Hr}.

We then show that every metric Lie algebra carrying a certain ``naturally reductive''-like decomposition is of Killing type. The details are given in Proposition \ref{lm:lemme}. As applications of this result we show that extensions by derivations and central extensions of Lie algebras endowed with ad-invariant metrics are of Killing type (Corollary \ref{cor:abideal} and Corollary \ref{cor:centext}).

Using these results and the classification of low-dimensional metric Lie algebras, we show in Section 5 that every metric Lie algebra of dimension 2 or 3 is of Killing type, and we obtain a similar result in Section 6 for two particular classes of Lie algebras of dimension 4 (those which are either non-solvable or have 1-dimensional derived ideal).

Based on the above results, there is perhaps not enough evidence in order to conjecture that all metric Lie algebras are of Killing type. However, all our attempts in order to construct a counterexample have failed so far. We hope to make further progress on this question in a subsequent work.

{\bf Acknowledgments.} This work was supported by the Procope Project No. 57445459 (Germany) /  42513ZJ (France).

\section{Conformal Killing tensors}

\subsection{Generalities}

We follow the notations from \cite{HMS17}. 
Let $V$ be a vector space of dimension $n$ endowed with a (positive definite) metric $g$. We define $\Sym^0V=\R$ and for $p\ge 1$ we denote by $\Sym^pV$ the subspace of symmetrized $p$-tensors on $V$:
\[
v_1\cdot\ldots\cdot v_p:=\sum_{\sigma\in \mathfrak S_p}v_{\sigma(1)}\otimes \ldots \otimes v_{\sigma(p)},
\]
where $v_i\in V$ and $\mathfrak S_p$ is the set of permutations of $\{1,\ldots,p\}$. In particular we have $v\cdot u=v\otimes u+u\otimes v$ for every $u,v\in V$. We further denote by $\Sym^*V$ the space of  all symmetric tensors on $V$, that is, $\Sym^*V=\bigoplus_{p\geq 0}\Sym^pV$.
  
The metric $g$ on $V$ induces a metric, also denoted by $g$, on $\Sym^pV$ as follows:
\[
g(v_1\cdot\ldots \cdot v_p,u_1\cdot\ldots \cdot u_p):=\sum_{\sigma\in \mathfrak S_p}\prod_{i=1}^ng(v_i,u_{\sigma(i)}).
\]
This allows us to identify every symmetric tensor $K\in \Sym^pV$ with a multilinear symmetric map on $V$, also denoted by $K$, through the identity
\begin{equation}\label{eq:identif}
K(v_1, \ldots, v_p):=g(K,v_1\cdot \ldots\cdot v_p), \qquad \mbox{ for all }v_1, \ldots, v_p\in V.
\end{equation}
Given $v\in V$ and $K\in \Sym^pV$, we denote by $v\lrcorner K$ the contraction of $K$ with $v$, that is, $(v\lrcorner K)(v_1, \ldots, v_{p-1}):=K(v,v_1, \ldots, v_{p-1})$, for every $v_1, \ldots, v_p\in V$. The linear maps 
\[
\begin{array}{rcl}
v\cdot \, :\Sym^pV&\lra& \Sym^{p+1}V\\
K&\mapsto& v\cdot K
\end{array}
\quad \mbox{ and } \quad
\begin{array}{rcl}
v\lrcorner :\Sym^pV&\lra& \Sym^{p-1}V\\
K&\mapsto& v\lrcorner K
\end{array}
\]
are adjoint to each other with respect to the above defined metric on $\Sym^pV$.

As usual, symmetric endomorphisms of $V$ are identified with symmetric bilinear maps on $V$ through the metric $g$, and thus with symmetric tensors in $\Sym^2V$ via \eqref{eq:identif}.
Given a symmetric endomorphism $K$ of $V$, the corresponding symmetric 2-tensor in $\Sym^2V$, also denoted by $K$, is  
\begin{equation}\label{endsym}
K=\frac12 \sum_{i=1}^n  K e_i\cdot e_i,
\end{equation}
 where $\{e_i\}$ is an orthonormal basis of $V$. 
If $\LL$ denotes the symmetric tensor $\sum_{i=1}^ne_i\cdot e_i$, the symmetric endomorphism and the symmetric bilinear form corresponding to $\LL$ are $2\mathrm{Id}$ and $2g$ respectively.

Given $M\in \End(V)$, we denote by $M^*$ its metric adjoint. The symmetric part of $M$ has its corresponding symmetric tensor, which we denote by $S_M\in \Sym^2V$. From \eqref{endsym} we have
\begin{equation}\label{eq:SM}
S_M=\frac12\sum_{i=1}^n\frac12(M+M^*)e_i\cdot e_i=\frac12\sum_{i=1}^n M e_i\cdot e_i.
\end{equation}
In addition, the endomorphism $M$ extends as a derivation of the algebra $\Sym^pV$. In particular, on decomposable symmetric 2-tensors, it satisfies
\[
M(u\cdot v)=Mu\cdot v+u\cdot Mv.
\]
If $K$ is a symmetric endomorphism of $V$, then the action of $M$ on the corresponding symmetric tensor $K$ is 
\[
M\left(\frac12 \sum_{i=1}^n Ke_i\cdot e_i\right)=\frac12 \sum_{i=1}^n (MKe_i\cdot e_i+Ke_i\cdot Me_i)=\sum_{i=1}^n MKe_i\cdot e_i=2S_{MK}.
\]
Consequently, this action, viewed in $\End(V)$, reads $M(K)=(MK+KM^*)$. So if $M$ is symmetric, $M(K)=(MK+KM)$ and if $M$ is skew-symmetric, $M(K)=[M,K]$.

In particular, for $\LL=\sum_{i=1}^ne_i\cdot e_i$ we have $M(\LL)=4S_M$. Since $M$ acts as a derivation, we further get
\begin{eqnarray}
M(\LL\cdot K)&=&4S_M\cdot K+\LL\cdot  M(K), \quad \mbox{ for every } K\in \Sym^pV.
\end{eqnarray}
This implies, by immediate induction, that for every $j\geq 1$ and every  $K\in \Sym^pV$,
\begin{equation}\label{eq:ML}
[M,\LL^j]\cdot K =4jS\cdot \LL^{j-1}\cdot K.
\end{equation}

The multiplication by $\LL$ in $\Sym^pV$ induces a linear operator, which we denote also by $\LL$, namely
\begin{equation*}
\LL:\Sym^pV\lra \Sym^{p+2}V, \qquad K\mapsto \sum_{i=1}^ne_i\cdot e_i\cdot K.
\end{equation*}

The contraction with the metric gives rise to the linear map
\begin{equation*}
\Lambda:\Sym^pV\lra \Sym^{p-2}V, \qquad K\mapsto \sum_{i=1}^ne_i\lrcorner e_i\lrcorner K.
\end{equation*}
Notice that $\Lambda$ vanishes on $\Sym^0V$ and on $\Sym^1V\simeq V$. Moreover, if $K$ is a symmetric endomorphism of $V$, viewed as an element in $\Sym^2V$ as in \eqref{endsym}, then $\Lambda K=\tr K$. We define, for each $p\geq 0$, $\Sym_0^pV:=\ker(\Lambda:\Sym^pV\lra \Sym^{p-2}V)$ as the subspace of  {\em trace-free} symmetric tensors. 

Since the contraction and symmetric product by vectors are metric adjoints, it follows that $\Lambda$ and $\LL$ are adjoint operators. This implies that every $K\in \Sym^pV$ decomposes uniquely as
\begin{eqnarray}\label{eq:decK}
K=K_0+\LL R, \qquad \mbox{ where } \quad \Lambda K_0=0 \mbox{ and }R\in \Sym^{p-2}V.
\end{eqnarray}
We call $K_0$ the trace-free part of $K$.

Finally, we consider the linear map $\deg :\Sym^* V\lra \Sym^* V$ which on symmetric $p$-tensors $K$ verifies $\deg K= p K$. One can easily  check that the operators previously defined verify the commutation law
\begin{equation}
[\Lambda,\LL]=2n\Id+4\deg\label{eq:LL}.
\end{equation}

Let now $(M,g)$ be a Riemannian manifold with Levi-Civita connection $\nabla$. The linear maps $\LL$, $\Lambda$ and $\deg$ extend to sections of the vector bundles $\Sym^p \mathrm{T}M$ for all $p\ge0$. 

We shall consider further geometric operators on these vector bundles.  Let $\{e_i\}_{i=1}^n$ denote a local orthonormal frame and define the symmetrized covariant derivative
\begin{equation}
\label{eq:dddefi}
\dd:\Gamma(\Sym^p\mathrm{T}M)\lra \Gamma(\Sym^{p+1}\mathrm{T}M), \qquad K\mapsto \dd K=\sum_{i=1}^ne_i\cdot \nabla_{e_i}K,
\end{equation}
and its formal adjoint
\begin{equation}
\label{eq:deltadefi}
\delta:\Gamma(\Sym^p\mathrm{T}M)\lra \Gamma(\Sym^{p-1}\mathrm{T}M), \qquad K\mapsto \delta K=-\sum_{i=1}^ne_i\lrcorner \nabla_{e_i}K.
\end{equation}
These operators are related to the linear operators $\Lambda$ and $\LL$ previously defined as follows:
\begin{eqnarray}
&[\Lambda,\delta]=0=[\LL,\dd],\quad[\Lambda,\dd]=-2\delta.\label{eq:LD}&
\end{eqnarray}

We are now in position to introduce the objects which are the subject of study of the paper.
\begin{defi}
Let $(M,g)$  be a Riemannian manifold. A symmetric $p$-tensor $K\in \Gamma(\Sym^p\mathrm{T}M)$ is a {\em Killing tensor} if $\dd K=0$, and a {\em conformal Killing tensor} if $\dd K=\LL B$ for some symmetric tensor $B\in \Gamma(\Sym^{p-1}\mathrm{T}M)$. 
\end{defi}
One can easily check that (conformal) Killing 1-tensors on $(M,g)$ correspond to (conformal) Killing vector fields.

\begin{remark}\label{rem:killing}
A symmetric tensor is conformal Killing if and only if its trace-free part is conformal Killing. Indeed, given a symmetric $p$-tensor $K$  on $(M,g)$ with decomposition $K=K_0+\LL R$ as in  \eqref{eq:decK}, then, since $\dd$ and $\LL$ commute, we have
\[
\dd K= \dd K_0 +\LL \dd R.
\]
Hence $\dd K= \LL B$ for  some symmetric $p-1$-tensor $B$ if and only if $\dd K_0=\LL(B-\dd R)$.
\end{remark}

From \cite{HMS17} we have a necessary and sufficient condition for a trace-free symmetric tensor to be conformal Killing:
\begin{pro}[\cite{HMS17}]\label{K0cK}
A trace-free symmetric tensor $K_0$ is conformal Killing if and only if $\dd K_0=a_0\LL \delta K_0$, where $a_0:=-\frac{1}{n+2p-2}$.
\end{pro}

It is always possible to construct conformal Killing tensors from a Killing tensor. Indeed, given a Killing $p$-tensor  $K$ and  a symmetric $p-2$-tensor $R$, one has $$\dd( K+\LL R)=\LL\dd R,$$ so that $K+\LL R$ is conformal Killing. 
Loosely speaking, the conformal Killing tensors constructed in this way are {\em Killing  up to the image of $\LL$}. Clearly, not every conformal Killing tensor is of this form. For instance, conformal Killing vector fields on  $(M,g)$ which are not Killing, are not of the form $K+\LL R$ with $K$ Killing.  We introduce the following definition:

\begin{defi}
A conformal Killing $p$-tensor $K$ is {\em of Killing type} if there exist a symmetric tensor $R$ such that $K + \LL R$ is Killing. 
\end{defi}

We show in the next result, among other equivalences, that a similar relation as the one mentioned above holds in this case: a symmetric Killing tensor is of Killing type if and only if its trace free part is of Killing type too.

\begin{pro}\label{pro:extK0}
Let $K$ be a conformal Killing symmetric $p$-tensor and let $B$ be the unique $p-1$-tensor satisfying $\dd K=\LL B$. If $K_0$ denotes the trace-free part of $K$, the following statements are equivalent:
\begin{enumerate}
\item \label{it1} $K$ is of Killing type, 
\item \label{it2} $K_0$ is of Killing type,
\item \label{it3} $B\in \Im (\dd)$,
\item \label{it4} $\delta K_0\in \Im(\dd)$.
\end{enumerate} 
\end{pro}
\begin{proof}
Using \eqref{eq:decK}, it is easy to check that \eqref{it1} and \eqref{it2} are equivalent. Next, since $\dd$ and $\LL$ commute, for every $R\in \Gamma(\Sym^{p-2}\mathrm{T}M)$  we have
\[
\dd(K+\LL R)=\LL (B+\dd R).
\]
As $\LL$ is injective, this implies that $K+\LL R$ is Killing if and only if $B=-\dd R$, thus proving  that \eqref{it1} is equivalent to \eqref{it3}.

Finally, we will show the equivalence between \eqref{it2} and \eqref{it4}. First notice that $K_0$ is conformal Killing by Remark \ref{rem:killing}, so Proposition \ref{K0cK} yields $\dd K_0=\LL(- \frac{1}{n+2p-2} \delta K_0)$. By the equivalence between \eqref{it1} and \eqref{it3} we get that $K_0$ is of Killing type if and only if $ \delta K_0\in \Im(\dd)$.
\end{proof}

\subsection{Conformal Killing tensors on Riemannian Lie groups}

In this section we describe the geometry of Riemannian Lie groups and study their left-invariant Killing and conformal Killing tensors.

Let $G$ be a connected Lie group endowed with a  Riemannian metric $g$ which is invariant under left translations, and let $\mg$ denote the Lie algebra of $G$. The metric $g$ is determined by its value on the tangent space at the identity, which we identify with $\mg$. Let $\nabla$ denote the Levi-Civita connection of $(G,g)$. Koszul's formula evaluated on left-invariant vector fields $X,Y,Z$ on $G$ reads
\begin{equation}\label{eq.Koszul}
\lela \nabla_X Y,Z\rira=\frac12( \lela [X,Y],Z\rira+\lela [Z,X],Y\rira+\lela [Z,Y],X\rira).
\end{equation}

From this formula we obtain in particular that the covariant derivative of two left-invariant vector fields is again left-invariant. We identify a left-invariant vector field $X$ with its value $x\in\mg$ at the identity, so that \eqref{eq.Koszul} becomes 
\begin{equation}\label{eq.Koszul1}
\nabla_xy=\frac12\left( [x,y]-\ad_x^*y-\ad_y^*x\right),
\end{equation}
where $\ad_x^*$ denotes the metric adjoint of $\ad_x$ with respect to $g$.

We are interested in studying left-invariant symmetric tensors on $(G,g)$ which satisfy the conformal Killing condition. To this purpose, we will consider the left-invariant sections of  $\Sym^p\mathrm{T}G$, which are determined by  their values at the identity. This is why, from now on, we will identify the space of left-invariant symmetric tensors on $(G,g)$ with elements in $\Sym^p\mg$. In particular, we will say that an element in $\Sym^p\mg$ is a (conformal) Killing tensor, if its corresponding left-invariant symmetric tensor on $(G,g)$ has this property.

The differential operators $\dd$ and $\delta$ of $(G,g)$ defined in \eqref{eq:dddefi} and \eqref{eq:dddefi} preserve the left-invariant sections of $\Sym^*\mathrm{T}G$, so they define algebraic operators, also denoted by $\dd:\Sym^p\mg \lra \Sym^{p+1}\mg$ and $\delta:\Sym^p\mg \lra \Sym^{p-1}\mg$. If $\{e_1, \ldots, e_n\}$ is an orthonormal basis of $\mg$, then
\begin{equation}
\label{eq:ddelta}
\dd K:=\sum_{i=1}^ne_i\cdot \nabla_{e_i}K,\qquad \delta K:=-\sum_{i=1}^ne_i\lrcorner \nabla_{e_i}K.
\end{equation}
Note that $\dd$ and $\delta$ are metric adjoints if and only if $\mg$ is unimodular.

For each $x\in \mg$, let $S_x$ denote the symmetric tensor   $S_{\ad_x}$ associated to $\ad_x$ as in \eqref{eq:SM}.
Using \eqref{eq.Koszul1}, we obtain:
\[
\dd x=\sum_{i=1}^n
e_i\cdot \nabla_{e_i}x=\frac12\sum_{i=1}^n
e_i\cdot (\ad_{e_i}x-\ad_{e_i}^*x-\ad_x^*e_i)=-\sum_{i=1}^n
e_i\cdot \ad_x e_i,
\]
whence
\begin{equation}\label{eq:symdiffx}
\dd x=-(\ad_x+\ad_x^*)=-2 S_x, \qquad \mbox{ for every }x\in \mg.
\end{equation}

\begin{ex}\label{ex:biinv}
Let $G$ be a connected Lie group with Lie algebra $\mg$ and let $g$ be a bi-invariant Riemannian metric on $G$, that is, such that left- and right-translations are isometries. Then the metric $g$ on $\mg$ is ad-invariant, namely
\begin{equation}\label{eq:adinv}
g(\ad_x y,z)+g(y,\ad_xz)=0,\qquad \mbox{ for every } x,y,z\in \mg.
\end{equation}
It is well known that in this case $\mg$ decomposes as an orthogonal direct sum of orthogonal ideals $\mg=\ms\oplus \mz$, where $\ms$ is a compact semisimple Lie algebra and $\mz$ is the center of $\mg$ (see for instance \cite{AB}).

For these metric Lie algebras, the operator $\dd$ vanishes identically. Indeed, equation \eqref{eq:adinv} states that $\ad_x$ is skew-symmetric with respect to $g$. Therefore, in view of \eqref{eq:symdiffx}, $\dd x=0$ for every $x\in \mg$ and thus vanishes on every symmetric $p$-tensor. In other words, every symmetric tensor in $\Sym^*\mg$ is a Killing tensor.
\end{ex}

The following result shows that on any Riemannian Lie group, the space of left-invariant conformal Killing tensors of degree $p\leq 2$ coincides with the space of Killing tensors of the given degree. For conformal vector fields this fact was already known (see \cite{TCX}); we include the proof here for the sake of completeness.

\begin{pro} \label{pro:degleq2} Let $K$ be a left-invariant  symmetric $p$-tensor on $(G,g)$ with $p\leq 2$. If $K$ is a  conformal Killing tensor then $K$ is a Killing tensor.
\end{pro}
\begin{proof}
Let $K$ be a left-invariant conformal Killing 1-tensor on $(G,g)$, that is $K$ is an element in $\mg$. Then there exists $c\in \R$ such that and $\dd K=c \LL$. 
From \eqref{eq:symdiffx}, we know that $\dd K=-2S_{K}$ so contracting $\dd  K$ twice with $K$ we get
\begin{eqnarray*}
2c |K|^2&=&-2g(S_K K,K)=-g(\ad_K^*K,K)=0.
\end{eqnarray*}
Therefore, either $c=0$ or $K=0$, and in both situations, $K$ is a Killing vector field.

Assume now that $K$ is a conformal Killing 2-tensor so that $\dd K=\LL B$ for some $B\in \mg$. For every $x\in \mg$ and making use of \eqref{eq:dddefi}, one has
\[
x\lrcorner x\lrcorner x\lrcorner \dd K=6g([x,Kx],x)\quad\mbox{ and }\quad
x\lrcorner x\lrcorner x\lrcorner(\LL \cdot B)=4|x|^2g(B,x).
\]
Therefore, the equality $\dd K=\LL B$ implies
\[g([x,Kx],x)=\frac23|x|^2g(B,x), \qquad \mbox{ for all }x\in \mg.\]
In particular, if $x$ is an eigenvector of the symmetric endomorphism $K$, the previous equation gives
\[
\,g(B,x) |x|^2= 0.
\] 
Since $K$ is symmetric and $\mg$ has a basis of eigenvectors of $K$, $B$ must be zero and thus $K$ is Killing.
\end{proof}

Despite this result, one should notice that for $p\ge 3$, there exist left-invariant symmetric $p$-tensors which are conformal Killing but not Killing. Trivial examples are provided by symmetric $p$-tensors of the form $\LL R$ where $R$ is an arbitrary symmetric $p-2$-tensor  which is not Killing. In fact, in this situation, $\dd (\LL R)=\LL\dd R\ne 0$ since the operator $\LL$ is injective. In the following example we construct a conformal Killing tensor which is neither Killing nor in the image of  $\LL$.

\begin{ex} Let $\mg$ be the Lie algebra of dimension six having an orthonormal basis $\{e_1, \ldots, e_6\}$ satisfying  the Lie bracket relations
\[
[e_1,e_2]=e_4,\quad [e_1,e_3]=e_5,\quad [e_2,e_3]=e_6.
\]
Then for each $i=4, \ldots, 6$, $e_i$ is in the center of $\mg$ so it defines a left-invariant Killing vector field and thus a Killing 1-tensor. Moreover, $T:=e_1\cdot e_6-e_2\cdot e_5+e_3\cdot e_ 4$ is a symmetric Killing 2-tensor on $\mg$ (see \cite{dBM}), that is, $\dd T=0$.

Since $\dd$ is a derivation, the symmetric 3-tensor $K:=T\cdot (e_4+e_5+e_6)$ is also a Killing tensor. It is easy to check that $\Lambda K=2(e_1-e_2+e_3)$ so $K$ is not trace-free. 

Consider the decomposition of $K=K_0+\LL R$ as in \eqref{eq:decK}; here $R$ is a vector in $\mg$. As pointed out in the previous subsection, the trace-free part $K_0$ is a conformal Killing 3-tensor. We claim that $K_0$ is not Killing. 

Indeed, $\dd K_0=-\LL \dd R$, so if $K_0$ were Killing, $R$ would be Killing too, which would imply $R\in \mz$. In addition, 
$\Lambda K=\Lambda \LL R$ which, by \eqref{eq:LD} and \eqref{eq:LL}, gives 
\[2(e_1-e_2+e_3)=16R\]
which clearly is not in the center $\mz$ of $\mg$, thus leading to a contradiction. Hence $K_0$ is conformal Killing but not Killing. Notice that, in addition, $K_0$ is of Killing type since $K=K_0+\LL R$ is Killing.
\end{ex}

In view of the above considerations, it makes sense to introduce the following notion:

\begin{defi}\label{def:kt} A metric Lie algebra $(\mg,g)$ is called {\em of Killing type} if every conformal Killing tensor is of Killing type.
\end{defi}

By the equivalence between \eqref{it1} and \eqref{it3} in Proposition \ref{pro:extK0}, and the fact that $\LL$ is injective and commutes with $\dd$, a metric Lie algebra $(\mg,g)$ is of Killing type if and only if $\Im(\dd)\cap\Im(\LL)=\Im(\LL\dd)$.

\section{Conformal Killing tensors on 2-step nilpotent Lie groups}

In this section we consider the case where the Riemannian manifold is a $2$-step nilpotent Lie group endowed with a left-invariant metric and we show that the corresponding metric Lie algebra is of Killing type.

Let $(N,g)$ be a Riemannian Lie group and let $\mn$ denote the Lie algebra of $N$. 
The center and the commutator of $\mn$ are, respectively,
\[\mz=\{z\in\mn\ |\   [x,z]=0, \,\mbox{for all }x\in\mn\},\qquad
\mn'=[\mn,\mn]:=\mathrm{span}\{ [x,y]\ |\   x,y\in\mn\}.
\]
The Lie algebra $\mn$ is said to be $2$-step nilpotent if it is non abelian and $\ad_x^2=0$ for all $x\in\mn$. 
Equivalently, $\mn$ is $2$-step nilpotent if $0\neq \mn'\subseteq \mz$.
 
For the rest of the section we assume that $\mn$ is $2$-step nilpotent and $N$ is a connected $2$-step nilpotent Lie group with Lie algebra $\mn$. We shall describe the main geometric properties of $(N,g)$ through linear objects in the metric Lie algebra $(\mn,g)$, following the work of Eberlein \cite{EB}.

Let $\mv$ be the orthogonal complement of $\mz$ in $\mn$ so that $\mn=\mv\oplus\mz$ as an orthogonal direct sum of vector spaces.  Each central element $z\in\mz$ defines an endomorphism $j(z):\mv\lra\mv$ by the equation
\begin{equation}\label{eq:jota}
\lela j(z)x,y\rira: =\lela z,[x,y]\rira \quad \mbox{ for all } x,y\in\mv.
\end{equation}
Each endomorphism $j(z)$ belongs to $\sso(\mv)$, the Lie algebra of skew-symmetric endomorphisms of $\mv$ with respect to $\bil$.

Many geometric features of the Riemannian manifold $(N,\bil)$ are captured by the linear map $j: \mz \lra \sso(\mv)$. In particular, by \eqref{eq.Koszul1} we readily obtain that the covariant derivative of left-invariant vector fields is expressed as follows:
\begin{equation}\label{eq:nabla}\left\{
\begin{array}{ll}
\nabla_x y=\frac12 \,[x,y] & \mbox{ if } x,y\in\mv,\\
\nabla_x z=\nabla_zx=-\frac12 j(z)x & \mbox{ if } x\in\mv,\,z\in\mz,\\
\nabla_z z'=0& \mbox{ if } z, z'\in\mz.
\end{array}\right.
\end{equation}
In particular, we have 
\begin{equation}\label{nablax} \nabla_x\mv\subset\mz,\qquad\nabla_x\mz\subset\mv,\qquad\forall x\in\mv,
\end{equation}
\begin{equation}\label{nablaz} \nabla_v\mv\subset\mv,\qquad\nabla_z\mz=0,\qquad\forall z\in\mz,
\end{equation}

The space of left-invariant symmetric Killing $p$-tensors on $(N,g)$ is identified with $\Sym^p\mn$. The decomposition $\mn=\mv\oplus\mz$  induces the following decomposition on the space of left-invariant symmetric tensors:
\begin{equation}
\Sym^p\mn=\bigoplus_{l=0}^p \Sym^l\mv\cdot \Sym^{p-l}\mz.
\end{equation}
Given $K\in \Sym^p\mn$, we denote $K^{(l)}$ its  component in $\Sym^l\mv\cdot \Sym^{p-l}\mz$ with respect to this decomposition and we call it as the component of $K$ of $\mv$-degree $l$.

For any symmetric $p$-tensor $K\in \Sym^l\mv\cdot \Sym^{p-l}\mz$, we  denote  $\deg_\mv(K)= lK$ and $\deg_\mz(K)= (p-l)K$, so that $\deg(K)=\deg_\mv(K)+\deg_\mz(K)$.

Using \eqref{nablax}, \eqref{nablaz} it is easy to prove that $\dd(\mv)\subset \mv\cdot\mz$ while $\dd(\mz)=0$. Recall that  $\dd$ is a derivation of $\Sym^p\mn$, so for every $p>0$ and $q\geq0$ the following inclusions hold:
\begin{equation}\label{eq:dd}
\dd ( \Sym^{q}\mz)=0 \qquad \mbox{ and } \qquad\dd (\Sym^p\mv\cdot \Sym^{q}\mz)\subset \Sym^p\mv\cdot \Sym^{q+1}\mz,
\end{equation}

As noticed above, for unimodular Lie algebras, in particular for nilpotent ones, the operator $\delta:\Sym^p\mn\lra \Sym^{p-1}\mn$ defined in \eqref{eq:ddelta} is the adjoint of $\dd$. Therefore, \eqref{eq:dd} implies:
\begin{equation}\label{eq:deltavz}
\delta (\Sym^{q}\mz)=\delta (\Sym^{q}\mv)=0 \; \mbox{ and } \;\delta (\Sym^p\mv\cdot \Sym^{q}\mz)\subset \Sym^p\mv\cdot \Sym^{q-1}\mz.
\end{equation}

Let $n_{\mv}$ and $n_{\mz}$ denote the dimensions of $\mv$  and $\mz$ respectively, so  that $n:=n_{\mv}+n_{\mz}$ is the dimension of $\mn$. Let $\{e_1, \ldots, e_{n_\mv}\}$ and $\{z_1, \ldots, z_{n_\mz}\}$ be orthonormal bases of $\mv$ and $\mz$, respectively. 
We define the following operators on $\Sym^p\mn$: for  every  $K\in \Sym^p\mn$,
\[\LL_\mv K:=\sum_{i=1}^{n_\mv} e_i\cdot e_i\cdot K,\qquad\LL_\mz K:=\sum_{i=1}^{n_\mz} z_i\cdot z_i\cdot K\] and 
\[\Lambda_\mv K:=\sum_{i=1}^{n_\mv} e_i\lrcorner e_i\lrcorner K, \qquad\Lambda_\mz K:=\sum_{i=1}^{n_\mz} z_i\lrcorner z_i\lrcorner K.\]
It is easy to check that these operators give a decomposition of 
 $\LL$ and $\Lambda$ as $\LL=\LL_\mv+\LL_\mz$ and $\Lambda=\Lambda_\mv+\Lambda_\mz$, and the following equalities hold
\begin{equation}\label{eq:LDvz}
[\Lambda_\mv,\LL_\mv]= 2n_\mv \Id +4 \deg_\mv, \quad [\Lambda_\mz,\LL_\mz]= 2n_\mz \Id +4 \deg_\mz,\;\mbox{ and }\;[\Lambda_\mz,\LL_\mv]=[\Lambda_\mv,\LL_\mz]=0.
\end{equation}

Recall that the left-invariant vector fields induced by elements in $\mz$ define Killing vector fields of $(N,g)$. So for every $i=1, \ldots, n_\mz$, $\dd z_i=0$ and hence
\[\dd (\LL_\mz K)=\sum_{i=1}^{n_\mz}\dd( z_i\cdot z_i\cdot K)= \sum_{i=1}^{n_\mz} z_i\cdot z_i\cdot\dd K=\LL_\mz\dd K.\] Moreover, $[\dd,\LL]=0$ (see \eqref{eq:LD}), therefore
\begin{equation}\label{eq:dLvz}
[\dd,\LL_\mz]=0=[\dd,\LL_\mv].
\end{equation}

Fix $l\in \Z_{\geq 0}$ and consider the real sequence $(b_r)_{r\geq 0}$ defined by
\[b_0=0, \qquad b_{r+1}=b_{r}-2n_\mv-4(l+2r), \mbox{ for } r\geq 0.\]
It is clear that $b_r<0$ for every $r\geq 1$.

\begin{pro}\label{pro:Hr}Let $K_0 \in \Sym^p\mn$ be a trace-free conformal Killing tensor. Then  for every $l=0, \ldots, p$  and $r\geq 0$, one has
\begin{equation}\label{eq:Hr}
\dd \Lambda^r_\mz K_0^{(l)}=a_0\left(
\LL_\mz\delta \Lambda^r_\mz K_0^{(l)}+\LL_\mv\delta \Lambda^r_\mz K_0^{(l-2)}+b_r \delta \Lambda^{r-1}_\mz K_0^{(l)}
\right),
\end{equation}
where $a_0=-\frac{1}{n+2p-2}$.
\end{pro}
\begin{proof}
We make the proof for each $l\in \{0, \ldots, p\}$ fixed, by induction on $r$. Since $K_0$ is conformal Killing, from Proposition \ref{K0cK} we have that $\dd K_0=a_0\LL \delta K_0$. Projecting this equality to $\Sym^{l}\mv\cdot \Sym^{p+1-l}\mz$, and using \eqref{eq:dd} and \eqref{eq:deltavz}, we obtain
\begin{equation}\label{dkl}
\dd K_0^{(l)}=a_0\LL_\mz\delta K_0^{(l)}+a_0\LL_\mv \delta K_0^{(l-2)}, \qquad l= 0, \ldots, p.
\end{equation}
Therefore \eqref{eq:Hr} holds for $r=0$ since $b_0=0$.

Now assume that \eqref{eq:Hr} holds for some $r\geq 0$ and apply $\Lambda$ to this equation. Using \eqref{eq:LD}, \eqref{eq:LL} and  \eqref{eq:LDvz}, we obtain
\begin{eqnarray*}
\dd\Lambda \Lambda^r_\mz K_0^{(l)}-2\delta \Lambda^r_\mz  K_0^{(l)}&=&a_0\left[(\LL_\mz\Lambda_\mv+\LL_\mz\Lambda_\mz+(2n_\mz \Id +4 \deg_\mz))\delta \Lambda^r_\mz  K_0^{(l)}\right.\\
&&\hspace{-3mm}\left.+(\LL_\mv\Lambda_\mz+\LL_\mv\Lambda_\mv+(2n_\mv \Id +4 \deg_\mv)) \delta\Lambda^r_\mz  K_0^{(l-2)}+b_r\delta \Lambda\Lambda^{r-1}_\mz K_0^{(l)}\right]\\
&=&a_0\left[\LL_\mz\delta\Lambda \Lambda^r_\mz  K_0^{(l)}
+(2n_\mz +4 (p-l-2r-1))\delta \Lambda^r_\mz  K_0^{(l)}\right.\\
&&\hspace{-3mm}\left.+\LL_\mv \delta\Lambda \Lambda^r_\mz  K_0^{(l-2)}
+(2n_\mv  +4 (l-2)) \delta \Lambda^r_\mz K_0^{(l-2)}+b_r\delta \Lambda\Lambda^{r-1}_\mz K_0^{(l)}\right].
\end{eqnarray*}
This is an equality between symmetric tensors, each of which is a sum of tensors of $\mv$-degrees $l$ and $l-2$.  Taking the projection of this equality onto $\Sym^l\mv\cdot \Sym^{p+1-l}$ we get
\begin{eqnarray*}
\dd\Lambda^{r+1}_\mz K_0^{(l)}-2\delta \Lambda^r_\mz K_0^{(l)}&=&a_0\left[\LL_\mz\delta\Lambda^{r+1}_\mz K_0^{(l)}+\LL_\mv \delta\Lambda_\mz^{r+1} K_0^{(l-2)})\right.\\
&&\left.\qquad
+(2n_\mz +4 (p-l-2r-1))\delta  \Lambda^r_\mz K_0^{(l)}+b_r\delta \Lambda^{r}_\mz K_0^{(l)}\right],
\end{eqnarray*}
and,  since $a_0=-\frac1{n+2p-2}$, this is equivalent to
\begin{eqnarray*}
\dd\Lambda^{r+1}_\mz K_0^{(l)}
&=&a_0\left[\LL_\mz\delta\Lambda^{r+1}_\mz K_0^{(l)}+\LL_\mv \delta\Lambda^{r+1}_\mz K_0^{(l-2)})
+(b_r-2n_\mv -4 (l+2r))\delta\Lambda^r_\mz  K_0^{(l)}\right].
\end{eqnarray*}
This proves that \eqref{eq:Hr} holds for $r+1$.
\end{proof}

\begin{cor}\label{cor:32}If $K_0$ is a trace-free conformal Killing $p$-tensor, then $\delta K_0\in \Im(\dd)$.
\end{cor}

\begin{proof} The result holds trivially for $p=0$ and, by Propositions \ref{K0cK} and \ref{pro:degleq2}, it also holds for $p=1,2$, so we assume $p\geq 3$. It is enough to show that $\delta K_0^{(l)}\in \Im(\dd)$ for each $l=0, \ldots, p$. 

Fix  a natural number $t$ such that $t\geq p/2$. We claim that for every  $s=0, \ldots, t$, 
\begin{equation}\label{eq:dts}
\delta\Lambda_\mz^{t-s}K_0^{(l)}\in \Im(\dd), \quad\mbox{ for every }l=0, \ldots, p.
\end{equation} 
Once this claim is  proved, and taking $s=t$, we will obtain $\delta K_0^{(l)}\in \Im(\dd)$ as we wanted to show.

Notice that for every $l=0, \ldots, p$, the inequality $p-l-2t\leq0$ holds. Hence $\Lambda^t_\mz K_0^{(l)}\in \Sym^l\mv$ and thus $\delta \Lambda^t_\mz K_0^{(l)}=0$ for every $l=0, \ldots, p$ by \eqref{eq:deltavz}. Therefore, \eqref{eq:dts} holds for $s=0$.

Suppose that  \eqref{eq:dts} holds for some $0\leq s<t$ and, for each $l=0, \ldots, p$,  let $B_{t-s}^{(l)}$ be a symmetric tensor such that $\delta\Lambda_\mz^{t-s}K_0^{(l)}=\dd B_{t-s}^{(l)}$. Using  \eqref{eq:Hr} for $r=t-s$ and the commutation of $\dd$ with $\LL_\mz$ and $\LL_\mv$, we obtain
\begin{eqnarray*}
a_0b_{t-s} \delta \Lambda^{{t-(s+1)}}_\mz K_0^{(l)}&=&-\dd \Lambda^{t-s}_\mz K_0^{(l)}+a_0\left(
\LL_\mz\delta \Lambda^{t-s}_\mz K_0^{(l)}+\LL_\mv\delta \Lambda^{t-s}_\mz K_0^{(l-2)}\right)\\
&=&-\dd \Lambda^{t-s}_\mz K_0^{(l)}+a_0\,\dd\left(
\LL_\mz B_{t-s}^{(l)}+\LL_\mv B_{t-s}^{(l-2)}\right).
\end{eqnarray*}
Since $a_0b_{t-s}\neq 0$ we obtain $\delta \Lambda^{{t-(s+1)}}_\mz K_0^{(l)}\in \Im(\dd)$.
\end{proof}

By Proposition \ref{pro:extK0}, Definition \ref{def:kt} and Corollary \ref{cor:32} we obtain at once:

\begin{teo} \label{teo:2nilp} Every $2$-step nilpotent metric Lie algebra is of Killing type.
\end{teo}

\section{Conformal Killing tensors on extensions of Lie algebras} \label{sec:centext}

In this section, we view symmetric Killing tensors as polynomials in a given orthonormal basis of the metric Lie algebra. We will decompose them with respect to a particular variable, corresponding to a specific direction inside the metric Lie algebra, assuming that the symmetric differential of vectors in this direction are polynomial in the remaining variables.

Let $(\mg,g)$ be a metric Lie algebra and let $t\in \mg$ be a unit vector. Denote by $E:= \left\langle t\right\rangle^\bot$ and let
\[\LL_E:=\sum e_i^2\in\Sym^2E\subset\Sym^2\mg,\]
where $\{e_i\}$ is  an orthonormal basis of $E$. Notice that $E$ is, in general, not a subalgebra of $\mg$.

The space of symmetric 2-tensors $\Sym^2\mg$ splits as orthogonal direct sum
\[
\Sym^2\mg=\Sym^2E\oplus t\cdot \mg,
\] and it is clear that $\LL_E\in \Sym^2E$ and $\LL=\LL_E+ t^2$. As before, we will also view $\LL_E$ as an operator on $\Sym^*\mg$ by $\LL_E(K):=\LL_E\cdot K$.

Using the fact that $\dd$ and $\LL$ commute, we have
\[
\dd\LL_E=\dd(\LL-t^2)=\LL \dd-2t\cdot \dd t-t^2\cdot \dd=\LL_E \dd-2t\cdot \dd t,
\]
so $[\dd,\LL_E]=-2t\cdot  \dd t$. By immediate induction, for any $k\geq 0$ we have: 
\begin{equation}
\label{eq:dLE}
[\dd,\LL_E^k]=-2k\, t\cdot \dd t \cdot \LL_E^{k-1}.
\end{equation}

Let $\{e_1, \ldots, e_n\}$ be an orthonormal basis of $E$. A symmetric $p$-tensor on $\mg$ can be viewed either as a polynomial in the variables $t,e_i$, $i=1, \ldots,n$ with real coefficients, or as a polynomial in the variable $t$ with coefficients in the ring of polynomials in $e_i$, $i=1, \ldots, n$. That is,
$\Sym^*\mg=\R[e_1, \ldots, e_n,t]=\R[e_1, \ldots, e_n][t]$.

We will now prove a technical result, which will have several useful applications later.
\begin{pro}\label{lm:lemme}Suppose that in $(\mg,g)$ there exists a unit vector $t\in \mg$ such that for $E:=\lela t\rira^\bot$ we have
$\ad_t(E)\subseteq E$ and for every $x,y,z\in E$, 
\begin{equation}\label{eq:xyz}
g(\ad_xy,z)+g(y,\ad_xz)=0.
\end{equation}
Then the metric Lie algebra $(\mg,g)$ is of Killing type. In addition, if $\dd t\neq 0$ then any Killing tensor on $(\mg,g)$ is of even degree in $t$.
\end{pro}

\begin{proof}
We first remark that the hypotheses imply that $\dd x\in t\cdot E$ for every $x\in E$. Indeed, given $x,y,z\in E$, by formula
\eqref{eq:symdiffx} we have
\[
\dd x(y,z)=-g(\ad_xy,z)-g(y,\ad_x z),
\]which is zero because of \eqref{eq:xyz}. Moreover, $\dd x(t,t)=-2g(\ad_xt,t)=2g(\ad_tx,t)=0$ since $\ad_t x\in E=\lela t\rira^\bot$. Therefore, $\dd x\in t\cdot E$ and thus there exists $M\in \End(E)$ such that \begin{equation}
\label{eq:dxtM}
\dd x=t\cdot Mx, \qquad \mbox{ for every } x\in E.
\end{equation}

We shall prove in addition, that $\dd t$ is a multiple of the symmetric 2-tensor $S_M$ corresponding to the symmetric part of $M$. Indeed, for $x,y\in E$, \eqref{eq:symdiffx} and \eqref{eq:dxtM} give
\[
g(Mx,y)=y\lrcorner t\lrcorner \dd x=-y\lrcorner t\lrcorner(\ad_x+\ad_x^*)=-g(\ad_xt,y)-g(t,\ad_xy).
\]Now using this equation and \eqref{eq:symdiffx} we obtain
\[g((M+M^*)x,y)=g(Mx,y)+g(x,My)=-g(\ad_xt,y)-g(\ad_yt,x)=-\dd t(x,y).\]
Therefore, 
\begin{equation}
\label{eq:diffT}
\dd t =-2S_M.
\end{equation}

Let $K$ be a symmetric tensor on $\mg$ such that $\dd K=\LL B$, for some symmetric tensor $B$ in $\mg$. In order to prove the first statement, we need to show that $B\in \Im(\dd)$.

We apply the division algorithm in $\R[e_1, \ldots, e_n][t]$ to write
\[
K=(\LL_E+t^2)\cdot P(t)+t\cdot T+C\]
where $P\in \R[e_1, \ldots, e_n][t]$ and $T,C\in \R[e_1, \ldots, e_n]$.
The conformal Killing condition for $K$ gives $\dd K=\LL \dd P+\dd(t\cdot T+C)=\LL B$, which implies $\dd (t\cdot T+C)=\LL \tilde B$, with $\tilde B:=B-\dd P$. We claim that, under the hypotheses above,  $\tilde B=0$ and thus $B\in \Im(\dd)$.

To prove this claim, notice that $\dd (t\cdot T+C)=\LL \tilde B$ gives
\[ \dd t\cdot T+t\cdot\dd T+\dd C=(\LL_E+t^2)\cdot\tilde B.\]
Equations \eqref{eq:dxtM} and \eqref{eq:diffT} imply that $\dd t\cdot T$ has zero degree in $t$ and $\dd C$ (resp. $t\cdot \dd T$) either vanishes or has degree 1 (resp. 2) in $t$; consequently, the polynomial $(\LL_E+t^2)\cdot\tilde B$ has at most degree 2 in $t$, so $\tilde B$ has zero degree in $t$ and the following system holds:
\begin{eqnarray}
\dd T&=&t \cdot\tilde B \label{eq:tBdT}\\
\dd C&=&0\\
\dd t\cdot T&=&\LL_E\cdot \tilde  B.\label{eq:LBdtT}
\end{eqnarray}
Notice that \eqref{eq:tBdT} and \eqref{eq:LBdtT} together imply 
\begin{equation}\label{eq:LET}
\LL_E\cdot\dd T=t\cdot \dd t \cdot T.
\end{equation}
From \eqref{eq:LBdtT} and the fact that $\LL_E$ is injective, it is  clear that $\dd t\cdot T= 0$ implies our claim $\tilde B=0$. So let us assume $\dd t\cdot T\neq 0$.

Since $T\neq 0$, there is some $k\geq 0$ such that $T=\LL_E^k \cdot T_0$ with $T_0\neq 0$ not divisible by $\LL_E$. Then using the commutation formula \eqref{eq:dLE}, we get
\begin{equation}\label{eq:aux}
\dd T=\dd (\LL_E^k\cdot T_0)=\LL_E^k\cdot \dd T_0-2kt\cdot\dd  t\cdot \LL_E^{k-1} \cdot T_0=\LL_E^{k-1}\cdot (\LL_E\cdot\dd T_0-2kt\cdot\dd  t\cdot T_0).
\end{equation}
Notice that \eqref{eq:dxtM} implies that $\dd  T_0=t\cdot M(T_0)$, since $T_0$ is a symmetric tensor in $E$. This fact, together with \eqref{eq:aux} and \eqref{eq:diffT}, allow us to write \eqref{eq:LET} as
\begin{equation}
\label{eq:red}
\LL_E \cdot M(T_0)=-2(2k+1) S_M\cdot T_0.
\end{equation}
Since $\LL_E$ does not divide $T_0$, we obtain that $\LL_E$ divides $S_M$ in the polynomial ring $\R[e_1, \ldots, e_n]$, and thus $S_M=a\Id=\frac{a}2\LL_E$, for some $a\in \R$. Hence, by \eqref{eq:red}, we get \[
M(T_0)=-a(2k+1) \cdot T_0.\]

Let $A$ denote the skew-symmetric part of $M$ and $l:=\deg T_0$. Then $M(T_0)=alT_0+A(T_0)$ so the last equation reads
\[
A(T_0)=-a(2k+1+l) T_0.
\]
Since $A$ is skew-symmetric, $A(T_0)=aT_0=0$ which gives $M(T_0)=0$, and thus $\dd T_0=0$. Then, by \eqref{eq:aux} we obtain $\dd T=0$, which by \eqref{eq:tBdT} gives $\tilde B=0$, as claimed. 

For the second part, assume that  $K$ is a Killing $p$-tensor so that $B$ above is zero, that is, $\dd K=0$. The first  part of the proof implies that $\tilde B=-\dd P=0$, and thus $T=0$ by \eqref{eq:LBdtT} since, by assumption, $\dd t\neq 0$. Hence $K=\LL\cdot P(t) +C$ where $P$ is a Killing tensor in $\Sym^*\mg$ and $C$ is  a Killing tensor in $\Sym^*E\subset \Sym^*\mg$.

Now we proceed by induction on the degree $p$ of $K$. Let  $\xi=at+\sum_{i=1}^na_i e_i$, with $a,a_i\in \R$ for $i=1, \ldots,n$, be a Killing vector field. By \eqref{eq:dxtM} we have
\[
0=\dd \xi=a\dd t+\sum_{i=1}^n a_i t\cdot Me_i.
\]Since $0\neq\dd t\in \Sym^2E$, this implies that $a=0$ and thus $K$ has zero degree in $t$. 

Suppose that every Killing $p-2$-tensor in $\mg$ has even degree on $t$ and let $K$ be a Killing $p$-tensor. Then, $K=\LL\cdot P+C$ with $P$ a Killing $p-2$-tensor and $C$ a Killing tensor of zero degree in $t$. Therefore, $\LL\cdot P=(\LL_E+t^2)\cdot P$ has even degree in $t$ and the same holds for $K$.
\end{proof}

\begin{remark}
The hypotheses in Proposition \ref{lm:lemme} can be interpreted by saying that the decomposition $\mg=\langle t\rangle\oplus E$ satisfies all conditions for being naturally reductive, except for the $\ad_t$-invariance of the metric on $E$. 
\end{remark}

We shall apply the previous result to two particular cases of metric Lie algebras.

\begin{cor} \label{cor:abideal} Let $(\mg,g)$ be a metric Lie algebra with a codimension $1$ ideal $\mh$, and let $h$ denote the restriction of $g$ to $\mh$. If the metric Lie algebra $(\mh,h)$ is such that $h$ is ${\rm ad_\mh}$-invariant, then $(\mg,g)$ is of Killing type.
\end{cor}
\begin{proof}
We  shall apply Proposition \ref{lm:lemme} to a unit vector $t$ spanning $\mh^\bot$ and $E:=\mh$. Since $\mh$ is an ideal of $\mg$, the condition $\ad_t(\mh)\subseteq \mh$ is automatically satisfied. Moreover, since $h$ is {\rm ad}-invariant on $\mh$, \eqref{eq:xyz} is satisfied. So the result follows by a direct application of the aforementioned proposition.
\end{proof}

The previous result in particular implies that if a Lie algebra $\mg$ possesses a codimension one abelian ideal, then $(\mg,g)$ is of Killing type, for any choice of the metric $g$.

Besides this application to Lie algebras possessing codimension one ideals, Proposition \ref{lm:lemme} can also be applied to certain Lie algebras with non-trivial center.

Let $(\mg,g)$ be a metric Lie algebra and let $t$ be a unit vector in the center of $\mg$. Let $p:\mg\lra \mh:=\mg/\lela t \rira$ be the quotient map, and denote by $p(x)=:\bar x$ for $x\in\mg$. We endow $\mh$ with the unique Lie algebra structure $[\;,\,]_\mh$ making $p$ a Lie algebra morphism.
The bilinear map
\begin{equation}\label{eq:omega}
\omega:\mh\times \mh\lra \R, \qquad \omega(\bar x,\bar y):=g([ x, y],t), \qquad \forall\ \bar x,\bar y\in\mh,
\end{equation}
is a well defined closed 2-form in $\mh$. 

Let $E$ denote the orthogonal of $t$ in $\mg$ and consider the metric $h$ on $\mh$ such that $p|_E:E\lra \mh$ is an isometry. We consider the Lie algebra $ \R t\oplus_\omega \mh$, whose underlying vector space is $ \R t\oplus \mh$ and with Lie bracket verifying:
\begin{equation}
\label{eq:brom}
[x,y]=[x,y]_\mh+\omega(x,y)t, \quad [x,t]=0\qquad \mbox{ for all } x,y\in\mh.
\end{equation}
 Moreover, we consider in $ \R t\oplus_\omega \mh$ the metric $\tilde h$ extending $h$ such that $\mh$ is orthogonal to $t$ and $t$ has unit length.
Then the linear map defined by
\[
f:\mg\lra \R t\oplus \mh,\qquad f(x)=\bar x +g(x,t)t,
\] is a Lie algebra isomorphism and an isometry from $(\mg,g)$ to $(\R t\oplus \mh,\tilde h)$.

In these notations we have:
\begin{cor}\label{cor:centext}
Let $(\mg,g)$ be a metric  Lie algebra with non-trivial center and assume there is a central vector of unit length $t$ such that the metric induced on the Lie algebra $(\mh,[\;,\,]_\mh)$ defined above is $\ad_\mh$-invariant. Then $(\mg,g)$ is of Killing type.
\end{cor}

\begin{proof}
Since $t$ is in the center of $\mg$, we have $\dd t=0$ by \eqref{eq:symdiffx}. So, in particular, $\dd t\in \Sym^2 E$, for $E:=\lela t\rira^\bot$.
The fact that the metric $h$ in $\mh$ is ad-invariant implies that \eqref{eq:xyz} holds, since the projection $p:E\lra \mh$ is an isometry. Therefore, the result follows from Proposition \ref{lm:lemme}.
\end{proof}

Finally, we shall consider Lie algebras which are direct sum of two orthogonal ideals, one of which is one-dimensional.

\begin{pro}\label{pro:exttrivial} Let $(\mg,g)$ be a metric Lie algebra which can be decomposed as an orthogonal direct sum of ideals as $\mg=\R t\oplus \mh$. Then $(\mg,g)$ is of Killing type.
\end{pro}
\begin{proof}
The Lie algebra $\mg$ is a direct sum of  orthogonal ideals $\R t$ and $\mh$, and we may assume that $t$ is of unit length. 

By \eqref{eq:dddefi}, the differential $ \dd$ on vectors  of $(\mg,g)$ verifies
\begin{equation}
\dd t=0, \qquad \dd x=\dd_0x, \qquad \forall x\in \mh,\label{eq:dd1}
\end{equation}
where, if $h$ denotes the metric induced by $g$ on $\mh$,  $\dd_0$ is the symmetric differential of $(\mh,h)$.
We  shall consider  an orthonormal basis $\{e_1, \ldots, e_n\}$  of $(\mh,h)$. 

Let $K$ be a conformal Killing tensor in $(\mg,g)$, and decompose it, as an element in $\R[e_1, \ldots, e_n][t]$, with respect to the variable $t$:
\[
K(t)=(\LL_0+t^2)\cdot P(t)+t\cdot T+C,
\]
where $P$ is a polynomial in $\R[e_1, \ldots, e_n][t]$ and $T,C$ are polynomials in $\R[e_1, \ldots, e_n]$. 

Then using \eqref{eq:dd1} and the decomposition of $K$, the equation $\dd K=\LL B$  becomes
\[
(\LL_0+t^2)\cdot \dd P(t)+t\cdot \dd_0 T+\dd_0 C =(\LL_0+t^2)\cdot  B,
\]
where $\LL_0$ is the operator induced by the metric in $\mh$.
Rearranging the terms in the last equation we get
\[
t\cdot \dd_0 T+\dd_0 C =(\LL_0+t^2)\cdot ( B-\dd P).
\]
The left-hand side of this equation is a polynomial of degree at most one in $t$, and the right-hand side has degree at least 2 on $t$, unless $B-\dd P=0$. This implies that $B=\dd P\in \Im(\dd)$ must hold.
\end{proof}

\section{Metric Lie algebras of dimension at most 3}

In this section we shall prove that every metric Lie algebra of dimension $\leq 3$ is of Killing type. 

Recall that on abelian Lie algebras, every metric is ad-invariant, so any abelian metric Lie algebra is of Killing type in virtue of Example \ref{ex:biinv}.

Moreover, any 2-dimensional Lie algebra is solvable, and a Lie algebra of dimension 3 is either simple or solvable. By a simple inspection of the isomorphism classes of 2- and 3-dimensional real Lie algebras (see \cite[I. 4]{JA}) one can check that every solvable Lie algebra of dimension 2 and 3 has a codimension one abelian ideal. In our context, this fact together with Corollary \ref{cor:abideal} imply the following.

\begin{pro}
Every solvable metric Lie algebra of dimension $2$ or $3$ is of Killing type.
\end{pro}

In the $3$-dimensional simple case, we shall refer to a result due to Milnor.
\begin{pro}\cite[\textsection 4]{Mil76}\label{pro:milbasis}
Any metric simple Lie algebra of dimension $3$ admits an orthonormal basis $\{x,y,z\}$ whose Lie brackets verify
\begin{equation}
[x,y]=az, \qquad [y,z]=bx,  \qquad [z,x]=cy,
\label{eq:brmilnor}
\end{equation}
for some non vanishing constants $a,b,c\in \R$.
\end{pro}

We  shall use the basis above to show that every simple metric Lie algebra of dimension $3$ is of Killing type.

Let $(\mg,g)$ be a metric simple Lie algebra of dimension 3 and let $\{x,y,z\}$ be an orthonormal basis as in Proposition \ref{pro:milbasis}.
Using \eqref{eq:symdiffx}, the differential map $\dd$ in the basis elements becomes
\begin{equation}\label{eq:dss}
\dd x=\alpha\, yz,\qquad \dd y=\beta\, xz,\qquad \dd z=\gamma \, xy,
\end{equation}
where $ \alpha:=c-a$, $\beta:=a-b$, $\gamma:=b-c$.

Since $\alpha+\beta+\gamma=0$, we have that either $\dd:\mg\lra \Sym^2\mg$ is injective or has a kernel of odd dimension. The former case corresponds to $a\neq b\neq c\neq a$, the latter can occur only when, up to reordering, $a=b\neq c$ or $a=b=c$.

\begin{pro} \label{pro:3dimsimple} Every metric simple Lie algebra of dimension $3$ is of Killing type. 
\end{pro}

\begin{proof} Let $g$ be a metric on a simple Lie algebra $\mg$ of dimension $3$ and let  $\{x,y,z\}$ be an orthonormal basis of $(\mg,g)$ as in Proposition \ref{pro:milbasis}.

We may apply Proposition \ref{lm:lemme} to $t:=x$ and its orthogonal, $E:={\rm span}\{y,z\}$. By \eqref{eq:brmilnor}, we have that $\ad_x (E)\subseteq E$ and $[E,E]\subseteq \lela x\rira$. The latter implies that Equation \eqref{eq:xyz} is verified. Therefore, the hypotheses of Proposition \ref{lm:lemme} are satisfied and thus $(\mg,g)$ is of Killing type.
\end{proof}

The rest of the section aims to describe the set of left-invariant Killing tensors of 3-dimensional metric simple Lie algebras such that $\dd:\mg \lra \Sym^2\mg$ is injective. This description will be used in the next section to study conformal Killing tensors on 4-dimensional metric Lie algebras.

Let $(\mg,g)$ be a 3-dimensional simple metric Lie algebra and let $\{x,y,z\}$ be an orthonormal basis as in Proposition \ref{pro:milbasis}, whose differentials are given by \eqref{eq:dss}. Let $J$ denote the  symmetric tensor 
\begin{equation}\label{eq:J}
J:=\alpha y^2-\beta x^2.
\end{equation}
Then $\dd J=2\alpha \beta xyz-2\beta\alpha xyz=0$ so $J$ is a Killing tensor on $(\mg, g)$. 

Let $\mcQ(J,\LL)$ be the set of polynomials in the symmetric 2-tensors $J$ and $\LL$, namely,
\[
\mcQ(J,\LL)=\{\sum_{i=1}^k\lambda_i J^{r_i}\LL^{s_i}\mid \lambda_i\in \R, \,r_i,s_i,k\in \Z_{\geq 0}\}.
\]
Since $J$ and $\LL$ are Killing and $\dd$ is a derivation of the algebra of symmetric tensors, every element in $\mcQ(J,\LL)$ is a Killing tensor in $(\mg,g)$. The following result shows that if $\dd$ is injective, then every symmetric Killing tensor in $(\mg,g)$ lies in $\mcQ(J,\LL)$.

\begin{pro}\label{pro:killdinj} Let $(\mg,g)$ be a $3$-dimensional simple metric Lie  algebra, and let $\{x,y,z\}$ be an orthonormal basis such that \eqref{eq:dss} holds. If $\alpha,\beta,\gamma$ are all non-zero, then every symmetric Killing tensor on $(\mg,g)$ is in $\mcQ(J,\LL)$.
\end{pro}

\begin{proof}
Notice that, since the differentials of the basis elements have the form \eqref{eq:dss}, the hypotheses of Proposition \ref{lm:lemme} hold for $t\in \{x,y,z\}$ and $E:=\lela t\rira^\bot$. Moreover, these differentials  are non-zero since $\alpha, \beta, \gamma$  do not vanish.

Therefore, by the second part of Proposition \ref{lm:lemme}, any symmetric Killing tensor in $\mg$ has even degree in each of the variables $x$, $y$, $z$. In particular, $\mg$ has no Killing vectors (i.e. Killing tensors  of degree 1), and every Killing $2$-tensor has the form $K=mx^2+ny^2+lz^2$, with $m,n,l\in \R$. Moreover, using \eqref{eq:dss}, $0=\dd K$  implies
\[
0=\alpha m+\beta n+\gamma l=\alpha m+\beta n-(\alpha+\beta) l,
\]and thus $m-l=\frac{(l-n)\beta}{\alpha}$. Making use of $\LL$ to rewrite $K$ as  $K=l\LL +(m-l)x^2+(n-l)y^2$, we obtain
\[
K=l \LL +\frac{(n-l)}{\alpha}(\alpha y^2-\beta x^2)=l\LL+\frac{(n-l)}{\alpha} J.
\]So every symmetric Killing 2-tensor on $(\mg,g)$ is in $\mcQ(J,\LL)$.

Now we proceed by induction. Suppose that every symmetric Killing $p-2$-tensor is  in $\mcQ(J,\LL)$, and let $K$ be a Killing $p$-tensor. By (the second part of the proof of) Proposition \ref{lm:lemme} applied to $t:=x$, we can write $K$ as
$K=\LL P+C$, where $P\in \R[x,y,z]$ and $C\in \R[y,z]$ are both Killing tensors of degree $p-2$ and $p$, respectively. The inductive hypothesis implies that $P\in \mcQ(J,\LL)$, and thus $\LL P\in \mcQ(J,\LL)$. We shall prove that $C=0$ and thus the result will follow.

Since $C$ is a symmetric $p$-tensor in $\R[y,z]$, we may write it as $C=\sum_{i=0}^p\lambda_i y^i z^{p-i}$, for some $\lambda_i\in \R$. Then
\[
0=\dd C=x\cdot \sum_{i=0}^p(i\beta \lambda_i  y^{i-1} z^{p-i+1}+(p-i)\gamma \lambda_i y^{i+1}z^{p-i-1}).
\]Hence the coefficients must satisfy
\[\lambda_p=\lambda_{p-1}=0,\; \mbox{ and }\;(i+1)\beta \lambda_{i+2}+(p-i)\gamma \lambda_i=0, \quad \mbox{ for }i=0, \ldots, p-2. \]
This implies $\lambda_i=0$ for all $i$, and thus $C=0$.
\end{proof}

\section{4-dimensional metric Lie algebras}

Let $(\mg,g)$ be a 4-dimensional metric Lie algebra.
Using Levi's decomposition theorem and the fact that the only semisimple Lie algebras of dimension $\leq 4$ are simple and 3-dimensional, we obtain that $\mg$ is either solvable or a direct sum of ideals $\mg=\ms\oplus \R$, with $\ms$ being 3-dimensional simple. In general, this decomposition is not orthogonal with respect to $g$. We will consider the two cases separately.

\subsection{Non-solvable case}

Let $g$ be a metric on a non-solvable 4-dimensional Lie algebra $\mg$. As mentioned above, $\mg$ is the direct sum of ideals $\mg=\ms\oplus \R$, so in particular, it has non-trivial center. Fix a unit vector $t$ in the center of $\mg$ and denote $E:=\lela t\rira^\bot$. As seen in Section \ref{sec:centext}, $(\mg,g)$ is isometrically isomorphic to the central extension $(\ms\oplus_\omega \R,g_0+g_\R)$, where $\omega$ is defined as in \eqref{eq:omega}, and $g_0$ is the metric induced on $\mg/\lela t\rira\simeq \ms$ such that the quotient map $p|_E:E \lra \mg/\lela t\rira$ is an isometry.

Let $\{x,y,z\}$ be an orthonormal basis of $(\ms, g_0)$ as in Proposition \ref{pro:milbasis}. 
We shall write $\omega$ in this basis as
\begin{equation}\label{eq:omegacan}\omega=p \,x\wedge y+q\, y\wedge z+r\, z\wedge x, \qquad p,q,r\in \R.\end{equation}

The symmetric differential $\dd_0$ of $(\ms,g_0)$ verifies \eqref{eq:dss}, with $\alpha+\beta+\gamma=0$. Recall that $J$ defined in \eqref{eq:J} and the symmetric tensor $\LL_0$  induced by $g_0$ are Killing tensor in $(\ms, g_0)$. We further denote by $\xi:=2(qx+ry+pz)$, where $p,q,r$ are the coefficients of $\omega$ in \eqref{eq:omegacan}.

We now view $\omega$ as an endomorphism of $\ms$ and extend it as a derivation to the algebra $\Sym^*\ms$. Since $\omega$ is skew-symmetric, we have $\omega(\LL_0)=0$ and straightforward computation gives $\omega( \xi)=0$.  In addition, using \eqref{eq:J} and \eqref{eq:omegacan} we obtain
\begin{eqnarray*}
\omega(J)&=&2\alpha y(-px+qz)-2\beta x(py-rz)\\
&=&-2p(\alpha+\beta) xy+2q\alpha yz+2r\beta xz\\
&=&2p\gamma xy+2q\alpha yz+2r\beta xz=2\dd_0(qx+ry+pz).
\end{eqnarray*}

These relations motivate the following definition. Let $\mathcal Q^p(\xi,J,\LL_0)\subseteq \Sym^p \ms$ denote the ring of symmetric $p$-tensors which are polynomials in $J,\LL_0$ and $\xi$.

\begin{lm} \label{lm:Qp} For every $p\ge 1$ and $S_p\in \mathcal Q^p(\xi,J,\LL_0)$ there exists $S_{p-1}\in \mathcal Q^{p-1}(\xi,J,\LL_0)$ such that $\omega( S_p)=\dd_0 S_{p-1}$.
\end{lm}
\begin{proof}
Since both $\omega$ and $\dd_0$ are derivations of  $\Sym^* \ms$, it is enough to prove the result for monomials. Let $m,n,l\in \Z_{\geq0}$ with $2m+2n+l=p$, so that $J^m\LL_0^n\xi^l\in \mcQ^p(\xi,J,\LL_0)$.

Recall that  $\dd J=\dd \LL_0=0$ and, as shown above, $\omega(J)=\dd_0 \xi$, $\omega( \xi)=0=\omega(\LL_0)$. Then either $m=0$, in which case $\omega(J^m\LL_0^n\xi^l)=0$, or $m\ge 1$ and we can write
\[
\omega(J^m\LL_0^n\xi^l)=m \omega(J) J^{m-1} \LL_0^n\xi^l =\frac{m}{l+1} \dd_0\left(J^{m-1} \LL_0^n\xi^{l+1}\right),
\] with $J^{m-1} \LL_0^n\xi^{l+1}\in \mathcal Q^{p-1}(\xi,J,\LL_0)$.
\end{proof}

\begin{lm}\label{lm:Ap} Assume that the coefficients  $\alpha, \beta,\gamma$ in \eqref{eq:dss} corresponding to the basis of  $(\ms,g_0)$ are all non-zero.
Let  $A_{p}$, $A_{p-1}$ be symmetric tensors in $(\ms,g_0)$ of degrees $p$ and $p-1$, respectively,  such that
\begin{equation}
\dd_0 A_{p-1}=\omega(A_p),\qquad 
\dd_0A_p=\LL_0 \omega(A_{p-1}).\label{Ap-1}
\end{equation}
Then there exist sequences of symmetric tensors $\{A_i\}_{i=0}^{p}$ and $\{S_i\}_{i=0}^p$ such that for every $i= 0, \ldots, p$, $S_i\in \mathcal Q^i(\xi,J,\LL_0)$ and
\begin{eqnarray}
\dd_0 A_{i-1}&=&\omega(A_i)\label{Ai}\\
\dd_0A_i&=&\dd_0 S_i-\LL_0 \omega(A_{i-1})\label{Ai-1}.
\end{eqnarray} 
In particular, $\omega(A_2)=\lambda\omega(J)$ for some $\lambda\in\R$.
\end{lm}
\begin{proof} Set  $S_p:=0$ which clearly belongs to $\mathcal Q^p(\xi,J,\LL_0)$. Then, by \eqref{Ap-1} we have that \eqref{Ai} and \eqref{Ai-1} are valid for $i=p$. Now we proceed by induction.

Suppose that \eqref{Ai} and \eqref{Ai-1} are valid for  a certain fixed value $i$. We shall define $A_{i-2}$ and $S_{i-1}$.

Equation \eqref{Ai-1} for $i$ implies that $\dd_0(A_i-S_i)=-\LL_0 \omega(A_{i-1})$ and thus $A_i-S_i$ is a conformal Killing tensor in $(\ms,g_0)$. By Proposition \ref{pro:3dimsimple} it is of Killing type, so $\omega(A_{i-1})\in \Im(\dd)$. That is, there exists a symmetric tensor $A_{i-2}$ such that $\dd_0 A_{i-2}=\omega (A_{i-1})$, so they verify \eqref{Ai}
for $i-1$.

Moreover, we obtain that $S_i-A_i-\LL_0 A_{i-2}$  is a Killing tensor in $(\ms,g_0)$. Since $\alpha, \beta,\gamma$ are non-zero, Proposition \ref{pro:killdinj} shows that $S_i-A_i-\LL_0 A_{i-2}=:K_i\in \mcQ(J,\LL_0)$. This together with \eqref{Ai-1} for $i$, imply
\[\dd_0A_{i-1}=\omega (A_i)=\omega(S_i-K_i)-\LL_0A_{i-2},
\]where $S_i-K_i\in \mathcal Q^i(\xi,J,\LL_0)$. Due to Lemma \ref{lm:Qp}, there exists $S_{i-1}\in \mathcal Q^{i-1}(\xi,J,\LL_0)$ such that $\omega(S_i-K_i)=\dd_0 S_{i-1}$. So  the previous equation becomes
\[\dd_0A_{i-1}=\dd_0 S_{i-1}-\LL_0A_{i-2},
\]
giving \eqref{Ai-1} for $i-1$.

To show the last claim for $A_2$, notice that \eqref{Ai-1} for $i=2$ implies that $A_2-S_2$ is a conformal Killing tensor in $(\ms,g_0)$ which, by Proposition \ref{pro:degleq2}, is in fact a Killing tensor. Therefore, $A_2-S_2\in \mcQ(J,\LL)$ and thus $A_2\in \mathcal Q^2(\xi,J,\LL)$. Hence, by Lemma \ref{lm:Qp}, we have $\omega(A_2)=\lambda \dd (\xi)=\lambda\omega(J)$, for some $\lambda\in \R$, since $\mathcal  Q^1(\xi,J,\LL)$ is spanned by $\xi$.
\end{proof}

Let $(\mg,g)$ be a $4$-dimensional metric Lie algebra. As pointed before, $(\mg,g)$ is isometrically isomorphic  to a central extension $(\ms\oplus_\omega \R,g_0+g_\R)$, where $(\ms,g_0)$ is a 3-dimensional metric simple Lie algebra. By \eqref{eq:brom}, the Lie bracket in $\mg$ satisfies:
\begin{equation}\label{eq:brg}
\ad_t=0, \mbox{ and  } \ad_w=\ad^\ms_w+\omega(w) \,t,\mbox{ for all }w\in \ms,
\end{equation}where $\ad^\ms$ denotes the adjoint representation of $\ms$. If $\dd$ and $\dd_0$ denote, respectively,  the differentials of $(\mg,g)$ and  $(\ms,g_0)$ on symmetric tensors, then the previous equation and  \eqref{eq:symdiffx} give
\begin{equation}\label{eq:dd0x}\begin{array}{rcl}
\dd w&=&-(\ad_w+\ad_w^*)=\dd_0w-\omega(x)\cdot w, \quad \mbox{ for all }w\in \ms,\\
\dd t&=&0.
\end{array}
\end{equation} 
These formulas imply that for every symmetric tensor $R$ on $\mg$, one has 
\begin{equation}\label{eq:dd0omega}
\dd(R)=\dd_0R-\omega(R)\cdot t.
\end{equation}

\begin{pro} Every non-solvable $4$-dimensional metric Lie algebra  is of Killing type.
\end{pro}
\begin{proof} 
Let $(\mg,g)$ be a $4$-dimensional non-solvable metric Lie algebra and consider the central extension $(\ms\oplus_\omega \R,g_0+g_\R)$ to which it is isometrically isomorphic. 

Let $\{x,y,z\}$ be an orthonormal basis of $(\ms,g_0)$ given in Proposition \ref{pro:milbasis}, so that the differential operator $\dd_0$ of $(\ms,g_0)$ satisfies \eqref{eq:dss}.

Notice that $\omega=0$ if and only if the metric Lie algebra $(\mg,g)$ is the  orthogonal direct product of $(\ms,g_0)$ and $(\R, g_\R)$. In this case, Proposition \ref{pro:exttrivial} shows that $(\mg,g)$ is of Killing type. So from now on we assume $\omega\neq 0$.
Also, if $\alpha,\beta,\gamma$ in \eqref{eq:dss} are all zero, then the metric of $(\ms,g_0)$ is ad-invariant. Therefore, the result holds by Proposition \ref{pro:exttrivial}.

Suppose now that only one of the coefficients is zero. Without loss of generality, we may assume $\gamma=0$ and $\beta=-\alpha\neq 0$; equivalently, $b=c$ in the basis given in Proposition \ref{pro:milbasis}. We claim that it is always possible to find an element $u\in \ms$ such that the hypotheses of Proposition \ref{lm:lemme} hold, which thus implies that $(\mg,g)$ is of Killing type.

If $\omega(z)=0$, we take $u:=x$ whose orthogonal space $E$ in $(\mg,g)$ is spanned by $y,z,t$.
Then $[x,t]=0$ and by \eqref{eq:brmilnor} and \eqref{eq:brg} we have
\[
[x,y]=az+\omega(x,y)t, \quad [x,z]=-by,  \;\mbox{ and }\;[E,E]\subseteq \lela x\rira.
\]
So $\ad_x$ preserves $E$ and \eqref{eq:xyz} is satisfied.

If $\omega(z)\neq0$, we fix $u:=\omega(z)=rx-qy$. In this case, $E:=\lela u\rira^\bot$ is spanned by $qx+ry$, $z$ and $t$. Again, $[u,t]=0$ and, using \eqref{eq:brmilnor},  \eqref{eq:brg} and the equality $b=c$, we obtain
\[
[\omega(z),qx+ry]=(r^2+q^2)(z+pt), \quad [\omega(z),z]=-b(qx+ry)-(r^2+q^2)t,
\]
so $\ad_{\omega(z)}$ preserves $E$, and moreover, $[E,E]\subseteq \lela \omega(z)\rira$. This proves our claim.

Finally, suppose that $\alpha,\beta,\gamma$ in \eqref{eq:dss} are all non-zero and let $\dd$ denote the differential operator of $(\mg,g)$ on symmetric tensors. Let $K$ be a conformal Killing tensor in $(\mg,g)$, with $\dd K=\LL B$ for some symmetric tensor $B$. Viewing $K$ as a polynomial in $\R[x,y,z][t]$ and using the division algorithm, we write as before
\[
K(t)=(\LL_0+t^2)\cdot P(t)+ t \cdot T+C,
\] where $P$ is viewed as a polynomial in the variable $t$ with coefficients in $\R[x,y,z]$, and $T$ and $C$ are polynomials in $x,y,z$, that is, symmetric tensors in $\ms$.
The conformal Killing equation give
\[
\LL B=\dd K=\dd \left(\LL P+t\cdot T+C\right)=\LL \dd P+t\cdot \dd T+\dd C,
\]which is equivalent to
\begin{equation}\label{eq:B0}
(\LL_0+t^2)\cdot  \tilde B=t\cdot \dd T+\dd C, 
\end{equation}
with $\tilde B:=B-\dd P$, a polynomial in $\R[t,x,y,z]$. 
Using  \eqref{eq:dd0omega} in this equation we further get
\[
(\LL_0+t^2)\cdot \tilde B=t\cdot \dd_0T-\omega(T)\cdot t^2+\dd_0 C -\omega(C)\cdot t.\]
The right hand side of this equality is a polynomial of degree at most 2 in $t$, hence $\tilde B$ is constant in $t$, that is $\tilde B\in \R[x,y,z]$.
Moreover, comparing the components with equal degree in $t$ gives rise to the system of equations on symmetric tensors in $\ms$:
\begin{eqnarray}
 \tilde B&=&-\omega(T),\\
\dd_0 T&=&\omega(C),\\
\LL_0 B_0&=&\dd_0C.\label{eq:LBdC}
\end{eqnarray}

In order to prove that $K$ is of Killing type, we need to show (by Proposition \ref{pro:extK0}) that $B$ is in $\Im(\dd)$ or, equivalently, that $B-\dd P=\tilde B=-\omega(T)$ is in $\Im(\dd)$. This is the last part of the proof.

We apply Lemma \ref{lm:Ap} to $A_p:=C$ and $A_{p-1}:=T$. This is possible because $T$ and $C$ verify the system above, so they clearly satisfy \eqref{Ap-1}. Therefore, there exist sequences of symmetric tensors $\{A_i\}_{i=0}^p$ and $\{S_i\}_{i=0}^p$,  such that $S_i\in\mcQ^i(\xi,J,\LL_0)$ and \eqref{Ai} and \eqref{Ai-1} hold, with the additional property $\omega(A_2)=\lambda \omega(J)$ for some $\lambda\in \R$.

Consider the symmetric $p-2$-tensor in $\mg$ defined as $R=\sum_{i=0}^{p-2}t^i\cdot  A_{p-i}$. Then, by \eqref{eq:dd0omega}, \eqref{Ai} and \eqref{Ai-1}, we have
\begin{eqnarray*}
\dd R&=&\sum_{i=0}^{p-2}t^i \cdot  \dd_0 A_{p-i}-\sum_{i=0}^{p-2}t^{i+1}\cdot \omega (A_{p-i})=\sum_{i=0}^{p-2}t^i\cdot  \dd_0 A_{p-i}-\sum_{i=1}^{p-1}t^{i}\cdot \omega (A_{p-i+1})\\
&=&\dd_0A_{p-2}-t^{p-1}\cdot \omega(A_{2})=\omega(A_{p-1})-t^{p-1}\cdot \omega(A_{2}).
\end{eqnarray*}
Recall that $\omega(A_2)=\lambda\omega(J)$, so 
\[\dd (\lambda t^{p-2}\cdot J)=\lambda t^{p-2} \dd_0J-\lambda t^{p-1}\cdot \omega(J)=-t^{p-1}\cdot\omega(A_2).
\]
 Therefore the previous equation reads $\dd R-\dd(\lambda t^{p-2}J)=\omega(A_{p-1})$ which, by definition of $A_{p-1}$, gives $\dd (R-\lambda t^{p-2}J)=\omega(T)\in \Im(\dd)$ and the result follows.
\end{proof}

\subsection{Solvable case}

Let $(\mg,g)$ be a metric solvable Lie algebra of dimension 4. The commutator ideal $\mg'=[\mg,\mg]$ of such Lie algebra is nilpotent and of dimension $\leq 3$. Due to the classification of nilpotent Lie algebras in small dimension, we know that $\mg'$ is either abelian or isomorphic to the Heisenberg Lie algebra of dimension $3$. 

If $\mg'$ is abelian and 3-dimensional, then Corollary \ref{cor:abideal} implies that $(\mg,g)$ is of Killing type. Other $4$-dimensional solvable Lie algebras may admit codimension one abelian ideals; this is the case when $\dim \mg'=1$ as the following result shows.

\begin{pro} Let $\mg$ be a $4$-dimensional solvable Lie algebra such that $\dim \mg'=1$, then $\mg$ has a codimension one abelian ideal. In particular, $(\mg,g)$ is of Killing type for any metric $g$ on $\mg$.
\end{pro}

\begin{proof}
Let $u$ be a unit vector spanning the commutator $\mg'$ of $\mg$. Then, there is a linear map $f:\mg\lra \R$ such that \[
[v,u]=f(v)u, \qquad \mbox{ for all }v\in \mg.
\]

Let $x,y,u$ be an orthonormal basis of $\ker f$ and let $z$ be a unit vector in $(\ker f)^\bot$. The possibly non-vanishing Lie brackets are
\[
[z,u]=f(z) u,\quad [x,y]=\alpha u, \quad [z,x]=\beta u, \quad [z,y]=\gamma u,
\]
where $\alpha,\beta,\gamma\in\R$. The Jacobi identity implies
\[
\alpha f(z) u=[z,[x,y]]=[[z,x],y]+[x,[z,y]]=0.
\]
If $\alpha=0$, then $\ker f$  is an abelian ideal of $\mg$. Otherwise, $f(z)=0$ and $u$ is a central element. In this case, define $v:= \gamma x-\beta y$ if at least one of $\beta$ or $\gamma$ are non-zero, and $v:= x$ if $\alpha=\beta=0$. It is easy  to verify that $\{u,v,z\}$  spans an abelian ideal in $\mg$. 

Therefore, Corollary \ref{cor:abideal} implies that for any metric $g$ in $\mg$, the metric Lie algebra $(\mg,g)$ is of Killing type.
\end{proof}

Our tools fail to show that the remaining metric solvable Lie algebras of dimension 4 are of Killing type. This concerns Lie algebras whose commutator has dimension 2, or dimension 3 and is isomorphic to the Heisenberg Lie algebra. However, we were not able to find conformal Killing tensors on such Lie algebras, which are not of Killing type.

\end{document}